\documentclass[12pt]{iopart}
\pdfoutput=1
\usepackage[utf8]{inputenc}
\usepackage{iopams}
\usepackage{amsthm,amssymb}
\usepackage{eucal}
\usepackage{mathrsfs} 
\usepackage{setstack}
\usepackage[pdftex]{graphicx,color}
\usepackage{natbib}

\newtheorem{conjecture}{Conjecture}

\newcommand{\BW}{\mathscr{B}}

\begin{document}

\title[Stably recoverable information for the ISP]{Explicit tight bounds on the\\stably recoverable information for the\\inverse source problem}
\author{M Karamehmedovi\'{c}}
\address{Department of Applied Mathematics and Computer Science, Technical University of Denmark, Matematiktorvet build. 303 B, DK-2800 Kgs. Lyngby, Denmark}
\ead{mika@dtu.dk}

\begin{abstract}
For the inverse source problem with the two-dimensional Helmholtz equation, the singular values of the 'source-to-near field' forward operator reveal a sharp frequency cut-off in the stably recoverable information on the source. We prove and numerically validate an explicit, tight lower bound for the spectral location of this cut-off. We also conjecture and support numerically a tight upper bound for the cut-off. The bounds are expressed in terms of zeros of Bessel functions of the first and second kind.
\end{abstract}

\submitto{\IP}

\maketitle

\makeatletter

\newtheorem{lemma}{Lemma}
\newtheorem{theorem}{Theorem}
\newtheorem{remark}{Remark}
\newtheorem{corollary}{Corollary}

\makeatother

\section{Introduction}\label{section:introduction}

We treat the single-frequency inverse source problem for the Helmholtz equation in the plane, illustrated in ~\fref{figure:Fig1}.
\begin{figure}
\centering
\scalebox{0.8}{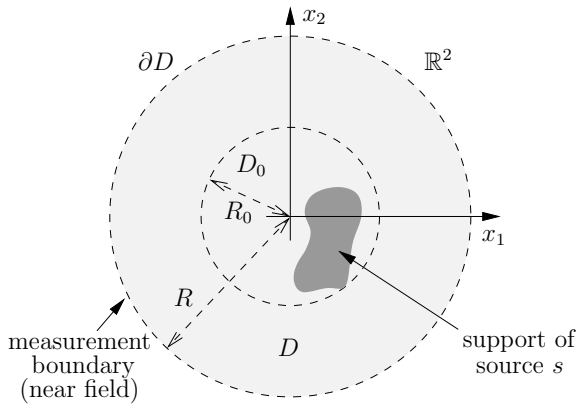}
\caption{Problem geometry and coordinate systems.}\label{figure:Fig1}
\end{figure}
Fix a positive constant wavenumber $k=2\pi/\lambda$, where $\lambda$ is the operating frequency, and let $D_0$ and $D$ be open disks in $\mathbb{R}^2$ centered at the origin and with radii $R_0$ and $R\ge R_0$, respectively. Write $\Delta=\partial_{x_1}^2+\partial_{x_2}^2$ for the Laplacian, and consider the Helmholtz problem
\begin{equation}\label{equation:IHP}
\left\{\begin{array}{rcl}
(\Delta+k^2)u&=&s\quad{\rm in}\,\,\,\mathbb{R}^2,\\
\lim_{|x|\rightarrow\infty}\sqrt{|x|}(\partial_{|x|}-\rmi k)u(x)&=&0,\quad{\rm uniformly}\,\,{\rm for}\,x/|x|\in S^1,
\end{array}
\right.
\end{equation}
for some source $s\in L^2(D_0)$ extended by zero to the whole plane. The second condition in~\eref{equation:IHP} is the outgoing Sommerfeld radiation condition in the plane. The inverse source problem, ISP, is now
\begin{center}
\textit{given a single measurement $U\in L^2(\partial D)$, find a source $s\in L^2(D_0)$ such that\\ there is a function ('radiated field') $u$ satisfying $u|_{\partial D}=U$ and satisfying the system~\eref{equation:IHP}}.
\end{center}
The ISP arises naturally in inverse acoustic and electromagnetic scattering, and has been devoted a substantial body of literature. The ISP is treated, e.g., in the multi-frequency regime by~\citet{Bao-2010}, and with far-field measurement data by~\cite{Griesmaier-2012}; see also~\cite{ElBadia-2011}. It occurs in antenna synthesis and diagnostics \citep{Persson-2005,Jorgensen-2010}, the analytic continuation of solutions of exterior scattering problems~\citep{SS,Zaridze-1998,Bliznyuk-2005,oac}, and in linearized inverse obstacle scattering problems.

In terms of the forward operator $F:s\mapsto U$, described in detail in~\sref{section:spectralanalysis}, solving the ISP amounts to solving
\begin{equation}\label{equation:Fs=U}
Fs=U\quad\text{for}\,\,s\in L^2(D_0).
\end{equation}
This problem is ill-posed, since $\ker F=(\Delta+k^2)H^2(D_0)$, where $H^2(D_0)$ is the Sobolev space $\{\partial^{\alpha}w\in L^2(D_0)\,\,\,{\rm for}\,\,\,\alpha\in\mathbb{N}^2_0\,\,\,{\rm with}\,\,\,|\alpha|\le2\}$. Also, measurements are typically noisy and sampled over a finite set of points. A common regularizing measure is to look for the minimum-$L^2$-norm, or minimum-energy, solution of~\eref{equation:Fs=U}, which is given by $s^{\dagger}=F^{\dagger}U$; here, $F^{\dagger}=(F^{\ast}F)^{-1}F^{\ast}=F^{\ast}(FF^{\ast})^{-1}$ is the Moore-Penrose pseudoinverse of $F$. Another regularization scheme uses a truncated singular value decomposition (TSVD) of the forward operator $F$. Here, $s$ is approximated by a finite sum of the form $\sum_1^N\sigma_m^{-1}(U,\phi)_m\psi_m$, with $(\sigma_m,\psi_m,\phi_m)$ a singular system of $F$. Our aim is to estimate the maximal amount of information about~\emph{any} source $s\in L^2(D_0)$ that can be stably recovered~\emph{in principle}, that is,~\textit{regardless} of the sampling frequency in the measurement and of the choice of the regularisation scheme. By 'stably recoverable information' we mean 'information recoverable robustly to noise,' and we refer to~\fref{figure:Fig2} for a more precise definition. A non-asymptotic analysis of the singular values $\sigma_m$ of the forward operator $F$, performed in~\sref{section:spectralanalysis}, reveals a low-pass filter behavior with well-defined passband and stopband. This turns out to be true also when the singular values are ordered according to increasing angular frequency of the right singular vectors of $F$, that is, of the singular vectors defined at the measurement boundary $\partial D$. In this case, the singular values within the passband generally do not increase or decrease monotonically, and the singular values in the stopband, still ordered according to angular frequency $m$, are monotonic functions of $m$.
\begin{figure}
\begin{center}
\scalebox{0.8}{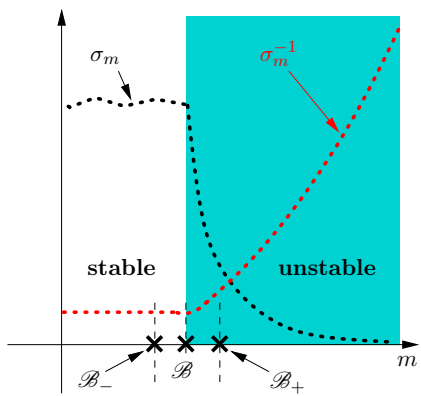}
\caption{A schematic of the singular value spectrum $\sigma_m$ and $\sigma_m^{-1}$ of the forward operator $F$ and its pseudoinverse $F^{\dagger}$, respectively, as function of angular frequency $m$ of right singular vectors of $F$. The lower bandwidth bound $\BW_-$ is given in Theorem 1. The upper bandwidth bound $\BW_+$ is predicted in Conjecture~\ref{conjecture:y}. Both $\BW_-$ and $\BW_+$ are validated numerically in section 3.}\label{figure:Fig2}
\end{center}
\end{figure}
We call the \emph{bandwidth} $\mathscr{B}$ of the forward operator $F$ the singular value index (angular frequency $m$ of a right singular vector of $F$) at which the singular value spectrum of $F$ becomes strictly decreasing as function of nonnegative $m$:
\begin{equation*}
\BW={\rm argmin}_{m\in\mathbb{N}_0}\{\sigma_{m+n}>\sigma_{m+n+1}\,\,{\rm for}\,\,{\rm all}\,\,n\in\mathbb{N}_0\}.
\end{equation*}
With this in mind, we define the stably recoverable information on a source $s$ to be the projection of $s$ onto the singular subspace of $F$ defined by $|m|\le\BW$. Then, finding the maximal amount of stably recoverable information about any source $s$, regardless of measurement sampling quality and of regularization scheme, amounts to estimating the bandwidth $\BW$ of the forward operator $F$.

To simplify the notation, write $\kappa_0=kR_0$ and $\kappa=kR$ for the size parameters of the source support and of the measurement boundary, respectively. Also, for integer $m$, write $j_{m,1}$ and $y_{m,1}$ for the first positive zero of the Bessel function $J_m$ of the first kind, respectively Bessel function $Y_m$ of the second kind, and order $m$. It is well-known~\citep[p. 146]{Magnus-1966} that $j_{m,1}>0$ for all $m\in\mathbb{N}_0$.
Our main result, proved in~\Sref{subsection:proofof}, is
\begin{theorem}\label{theorem:MAIN}
The bandwidth $\BW$ of the forward operator $F:s\mapsto U$ associated with the Helmholtz problem~\eref{equation:IHP} and measurement at $\partial D$ is bounded from below by
\begin{equation*}
\BW_-={\rm argmin}_{m\in\mathbb{N}_0}\{j_{m,1}\ge\kappa_0\}.
\end{equation*}
\end{theorem}
For convenience, in~\Sref{subsection:proofof} we also show that the bandwidth bound of Theorem~\ref{theorem:MAIN} can be expressed explicitly in the source size parameter $\kappa_0$:
\begin{corollary}\label{corollary:ceil}
For sufficiently large $\kappa_0$, we have
\[
\hspace{-25mm}\BW_-\approx\widetilde{\mathscr{B}}_-\hspace{-1mm}=\hspace{-1mm}\left\lceil\hspace{-1mm}\left(\frac{1}{6}\left(108\kappa_0+12\sqrt{12a_-^3+81\kappa_0^2}\right)^{1/3}\hspace{-1mm}-\frac{2a_-}{\left(108\kappa_0+12\sqrt{12a_-^3+81\kappa_0^2}\right)^{1/3}}\right)^3\,\right\rceil
\]
with $a_-=1.855757$.
\end{corollary}
Finally, the general form of the result in Theorem~\ref{theorem:MAIN}, as well as extensive numerical experimentation, lead us to conjecture a tight upper bound on the bandwitdth $\mathscr{B}$:
\begin{conjecture}\label{conjecture:y}
\begin{equation*}
\BW_+={\rm argmin}_{m\in\mathbb{N}_0}\{y_{m,1}\ge\kappa_0\}.
\end{equation*}
\end{conjecture}

In~\sref{section:spectralanalysis} we analyze the singular value spectrum of the forward operator $F$. In particular, we prove Theorem~\ref{theorem:MAIN} and Corollary~\ref{corollary:ceil} in~\sref{subsection:proofof}. We validate the bounds $\BW_-$ and $\BW_+$ on the bandwidth $\BW$ numerically in~\sref{section:numericalvalidation}, and discuss some implications of Theorem~\ref{theorem:MAIN} in~\sref{section:discussion}. A conclusion and suggestions for further work are given in~\sref{section:conclusionand}.

\section{Spectral analysis of the forward operator}\label{section:spectralanalysis}

The function $(\rmi/4)H_0^{(1)}(k|x|)$, $x\in\mathbb{R}^2$, is the radial outgoing fundamental solution of the Helmholtz operator in the plane, with singularity at the origin. Recall that $H_0^{(1)}=J_0+\rmi Y_0$ is the Hankel function of zero order and of the first kind. As in~\cite{Bao-2010}, introduce the forward operator
\begin{equation*}
Fs(x)=\int_{y\in D_0}H_0^{(1)}(k_0|x-y|)s(y),\quad x\in\partial D,\,s\in L^2(D_0),
\end{equation*}
that maps sources $s$ to the traces at $\partial D$ of the corresponding radiated fields. It is well-known~\citep{Bao-2010} that $F:L^2(D_0)\rightarrow L^2(\partial D)$ is compact. The adjoint $F^{\ast}$ is defined by
\begin{equation*}
F^{\ast}U(y)=\int_{x\in\partial D}H_0^{(2)}(k|x-y|)U(y),\quad y\in D_0,\,\,U\in L^2(\partial D),
\end{equation*}
where $H_0^{(2)}=J_0-\rmi Y_0$ is the Hankel function of zero order and of the second kind.

\subsection{A singular system of $F$}\label{subsection:asingular}

\citet{Bao-2010} derived a singular system of the forward operator $F$. We here slightly improve a part of their Proposition 2.1:
\begin{lemma}\label{lemma:SVD}
The forward operator $F$ admits the singular value decomposition
\begin{equation*}
F=\sigma_0(\cdot,\psi_0)_{L^2(D)}\phi_0+\sum_{m\in\mathbb{N}}\sigma_m\left[(\cdot,\psi_m)_{L^2(D)}\phi_m+(\cdot,\psi_{-m})_{L^2(D)}\phi_{-m}\right],
\end{equation*}
where
\begin{equation}\label{equation:xi_m}
\sigma_m=\sqrt{2R}\pi R_0|H_m^{(1)}(\kappa)|A_m(\kappa_0),\quad m\in\mathbb{N}_0,
\end{equation}
and
\begin{eqnarray*}
\psi_m(y)&=&(\sqrt{\pi}R_0A_m(\kappa_0))^{-1}J_m(k|y|)\rme^{\rmi m\arg y},\\
\phi_m(x)&=&(2\pi R)^{-1/2}\rme^{\rmi\arg H_m^{(1)}(\kappa)}\rme^{\rmi m\arg x},
\end{eqnarray*}
for $m\in\mathbb{Z}$, $x\in\partial D$ and $y\in D_0$.
Here
\begin{eqnarray*}
A_m(\kappa_0)&=&\sqrt{J_m(\kappa_0)^2-J_{m-1}(\kappa_0)J_{m+1}(\kappa_0)}\\&=&\sqrt{J_m(\kappa_0)^2+J_{m+1}(\kappa_0)^2-\frac{2m}{\kappa_0}J_m(\kappa_0)J_{m+1}(\kappa_0)}
\end{eqnarray*}
for $m\in\mathbb{Z}$.
\end{lemma}
Our slight improvement of Proposition 2.1 of~\citet{Bao-2010} consists in explicitly evaluating the integral $\int_{\varrho=0}^{R_0}\varrho J_m^2(k\varrho)$, occurring in $\sigma_m$ and $\psi_m$, in terms of $A_m(\kappa_0)$. This explicit evaluation is crucial to our proof of Theorem~\ref{theorem:MAIN}. We also note that our expressions for the singular vectors $\phi_m$, as well as the singular values $\sigma_m$, differ from~\cite{Bao-2010} in that they are only proportional to those given in that reference.

\begin{proof}[Proof of Lemma~\ref{lemma:SVD}]
For $s\in L^2(D_0)$ and $y\in D_0$ we have
\begin{equation}\label{equation:F*F}
F^{\ast}Fs(y)=\int_{z\in D_0}s(z)\int_{x\in\partial D}H_0^{(1)}(k|x-z|)H_0^{(2)}(k|x-y|).
\end{equation}
A special case of the Graf addition theorem \citep[Eq. 9.1.79, p. 363]{AbramowitzStegun} reads
\begin{equation*}
H_0^{(1)}(k|x-y|)=\sum_{m\in\mathbb{Z}}H_m^{(1)}(\kappa)J_m(k|y|)\rme^{\rmi m(\arg x-\arg y)},\quad x\in\partial D,\, y\in D_0.
\end{equation*}
Similar to~\cite{Bao-2010}, inserting this in~\eref{equation:F*F} we get
\begin{eqnarray*}
\hspace{-1cm}F^{\ast}Fs(y)&=\sum_{m,n\in\mathbb{Z}}H_m^{(1)}(\kappa)H_n^{(2)}(\kappa)J_n(k|y|)\rme^{-\rmi n\arg y}\\&\times\int_{z\in D_0}s(z)J_m(k|z|)\rme^{-\rmi m\arg z}\int_{x\in\partial D}\rme^{\rmi(m+n)\arg x}\\&=2\pi R\sum_{m\in\mathbb{Z}}|H_m^{(1)}(\kappa)|^2J_m(k|y|)\rme^{\rmi m\arg y}\int_{z\in D_0}s(z)J_m(k|z|)\rme^{-\rmi m\arg z},
\end{eqnarray*}
since $J_{-m}=(-1)^mJ_m$ and $Y_{-m}=(-1)^mY_m$ for all integer $m$. This gives an eigendecomposition of the operator $F^{\ast}F$; to normalize the eigenvectors, we note that \citet[Eq. 5.54.2, p. 629]{GradRyzh} gives
\begin{equation*}
\int\varrho J_m(k\varrho)^2=\frac{\varrho^2}{2}\left(J_m(k\varrho)^2-J_{m-1}(k\varrho)J_{m+1}(k\varrho)\right),\quad m\in\mathbb{Z},
\end{equation*}
and the recursion formula for cylinder functions \citep[Eq. 8.471.1, p. 926]{GradRyzh} implies
\begin{equation}\label{equation:recursion}
J_{m-1}(\kappa)+J_{m+1}(\kappa)=\frac{2m}{\kappa}J_m(\kappa),\quad m\in\mathbb{Z}.
\end{equation}
Thus,
\begin{eqnarray*}
\hspace{-20mm}\int_{\varrho=0}^{R_0}\varrho J_m(k\varrho)^2&=R_0^2\int_{\varrho=0}^1\varrho J_m(\kappa_0\varrho)^2=\frac{R_0^2}{2}\left(J_m(\kappa_0)^2-J_{m-1}(\kappa_0)J_{m+1}(\kappa_0)\right)\\&=\frac{R_0^2}{2}\left(J_m(\kappa_0)^2+J_{m+1}(\kappa_0)^2-\frac{2m}{\kappa_0}J_m(\kappa_0)J_{m+1}(\kappa_0)\right)=\frac{R_0^2A_m(\kappa_0)^2}{2},
\end{eqnarray*}
and $F^{\ast}F$ admits the spectral decomposition \begin{equation*}F^{\ast}F=\sigma_0^2(\cdot,\psi_0)_{L^2(D_0)}+\sum_{m\in\mathbb{N}}\sigma_m^2\left[(\cdot,\psi_m)_{L^2(D_0)}\psi_m+(\cdot,\psi_{-m})_{L^2(D_0)}\psi_{-m}\right].\end{equation*} Evidently, $\sigma_0^2$ has multiplicity one and all the other eigenvalues $\sigma_m^2$, $m\in\mathbb{N}$, have multiplicity two. The lemma now follows from Theorem 4.7 on p. 100 of~\citet{ColtonKress}; it here just remains to compute
\begin{eqnarray*}
\phi_m(x)&=&\sigma_m^{-1}F\psi_m(x)=\frac{\int_{y\in D_0}H_0^{(1)}(k|x-y|)J_m(k|y|)\rme^{\rmi m\arg y}}{\sqrt{2R}\pi^{3/2}R_0^2|H_m^{(1)}(\kappa)|A_m(\kappa_0)^2}\\&=&\frac{\sum_{\nu\in\mathbb{Z}}H_{\nu}^{(1)}(\kappa)\rme^{\rmi\nu\arg x}\int_{\varrho=0}^{R_0}\varrho J_{\nu}(k\varrho)J_m(k\varrho)\int_{\theta=0}^{2\pi}\rme^{\rmi\theta(m-\nu)}}{\sqrt{2R}\pi^{3/2}R_0^2|H_m^{(1)}(\kappa)|A_m(\kappa_0)^2}\\&=&(2\pi R)^{-1/2}\rme^{\rmi\arg H_m^{(1)}(\kappa)}\rme^{\rmi m\arg x},\quad x\in\partial D,\,\,m\in\mathbb{N}_0.
\end{eqnarray*}
\end{proof}

Figure 3 shows the first 71 nonnegative-index singular values of the forward operator $F$ with size parameters $\kappa=\kappa_0=10\pi$.
\begin{figure}\label{figure:Fig3}
\begin{center}
\includegraphics[width=0.52\textwidth]{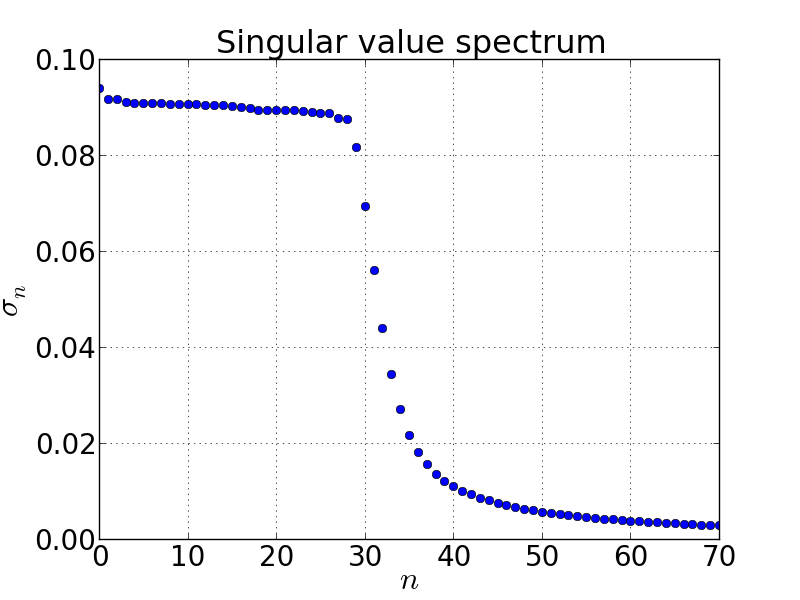}
\hspace{-8mm}\includegraphics[width=0.52\textwidth]{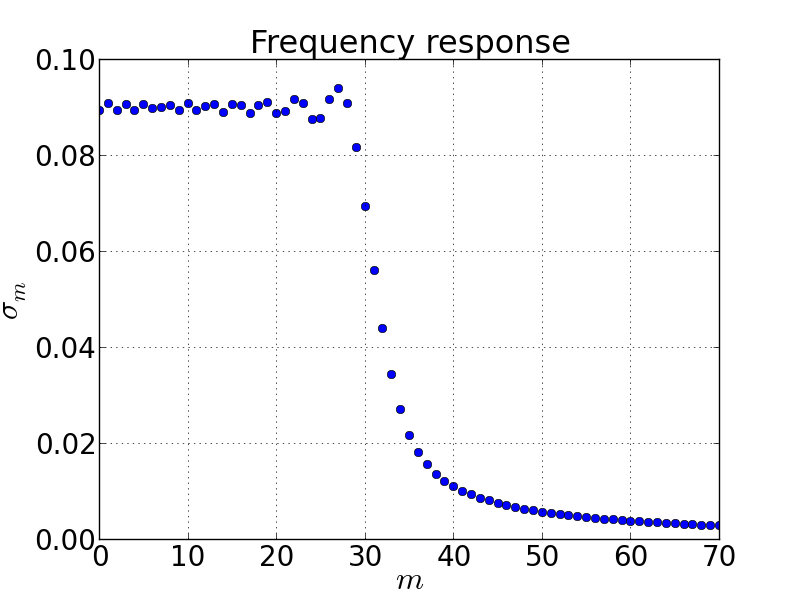}
\caption{Part of the singular value spectrum of the forward operator $F$ for $\kappa=\kappa_0=10\pi$. Left: the singular values $\sigma_n$, $0\le n\le70$, ordered in a decreasing sequence. Right: the same singular values ordered according to the angular frequency $m$ of the right singular vector $\phi_m$.}
\end{center}
\end{figure}
Clearly, the forward operator is a low-pass filter with respect to the singular values $\sigma_m$, with bandwidth $\BW=27$. We quantify the frequency response of this filter in~\sref{subsection:proofof}.

\subsection{Proof of Theorem~\ref{theorem:MAIN} and of Corollary~\ref{corollary:ceil}}\label{subsection:proofof}
In this section we prove the lower bound $\BW_-$ on the bandwidth $\BW$ given in Theorem~\ref{theorem:MAIN}, and the approximate value of $\BW_-$ given in Corollary~\ref{corollary:ceil}. For completeness, we first prove that the distance between the zeros of the function $0\le\mu\mapsto J_{\mu}(\kappa_0)$ is greater than $1$.
\begin{lemma}\label{lemma:only_one_zero}
If $\mu_2>\mu_1\ge0$ and $J_{\mu_1}(\kappa_0)=J_{\mu_2}(\kappa_0)=0$ then $\mu_2-\mu_1>1$.
\end{lemma}
\begin{proof}
Let $n\in\mathbb{N}$. The interlacing property of the zeros of Bessel functions (see, e.g.,~\citet{Palmai-2011}) implies $j_{0,n+1}>j_{1,n}$. Since $\mu\mapsto j_{\mu,n}$ is strictly increasing with $\mu$, $j_{\mu,n}=j_{0,n+1}$ implies $j_{\mu,n}>j_{1,n}$, hence $\mu>1$. As illustrated in~\fref{figure:Bessel_zeros}, to show that $j_{\mu_1,n+1}=j_{\mu_2,n}$ implies $\mu_2-\mu_1>1$, it now suffices to establish that
\begin{equation}\label{equation:ddmu}
\frac{dj_{\mu,n}}{d\mu}<\frac{dj_{\mu,n+1}}{d\mu}\quad\text{for all }\mu\ge0,\,n\in\mathbb{N}.
\end{equation}
\begin{figure}
\begin{center}
\includegraphics[scale=0.5]{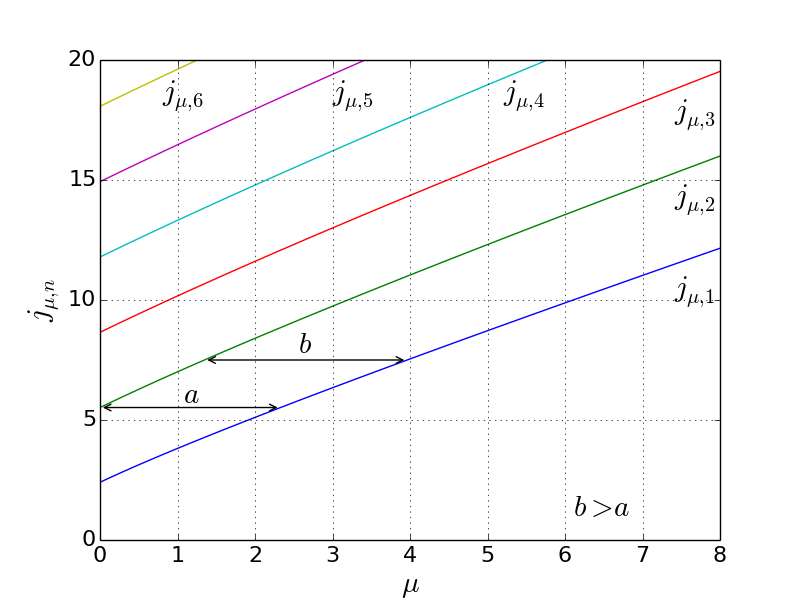}
\caption{The zeros of the function $0\le\mu\mapsto J_{\mu}(\kappa_0)$ diverge, see Lemma~\ref{lemma:only_one_zero}.}\label{figure:Bessel_zeros}
\end{center}
\end{figure}
For nonnegative order $\mu$, the $n$'th zero $j_{\mu,n}$ of the Bessel function $J_{\mu}$ satisfies~\citep[pp. 508--510]{Watson-1945}
\[
\frac{dj_{\mu,n}}{d\mu}=2j_{\mu,n}\int_{t=0}^{\infty}K_0(2j_{\mu,n}\sinh t)\rme^{-2\mu t}.
\]
Substituting $q=2j_{\mu,n}\sinh t$ and using that $j_{\mu,n}>0$, $\exp({\rm arcsinh}\,\tau)=\tau+\sqrt{1+\tau^2}$, as well as that $\cosh{\rm arcsinh}\,\tau=\sqrt{1+\tau^2}$, we get
\[
\frac{dj_{\mu,n}}{d\mu}=\int_{q=0}^{\infty}K_0(q)\left(q/2j_{\mu,n}+\sqrt{1+q^2/4j_{\mu,n}^2}\right)^{-2\mu}\left(1+q^2/4j_{\mu,n}^2\right)^{-1/2}.
\]
Setting
\[
f(c)=\int_{q=0}^{\infty}K_0(q)(q/2c+\sqrt{1+q^2/4c^2})^{-2\mu}\left(1+q^2/4c^2\right)^{-1/2},\quad c>0,
\]
we find
\[
\frac{\partial f}{\partial c}(c)=2^{2\mu+1}c^{2\mu}\int_{q=0}^{\infty}\frac{qK_0(q)\left(2\mu\sqrt{4c^2+q^2}+q\right)}{(4c^2+q^2)^{3/2}\left( \sqrt {4c^2+q^2}+q \right)^{2\mu}}>0\quad{\rm for}\,\,c>0.
\]
Finally, for any $\mu\ge0$ and $n\in\mathbb{N}_0$, we have $j_{\mu,n}<j_{\mu,n+1}$, so~\eref{equation:ddmu} indeed holds.
\end{proof}
We can now link the variation of the function $m\mapsto A_m(\kappa_0)$ with that of the Bessel function of the first kind. Fix $m\in\mathbb{N}_0$.
\begin{lemma}\label{lemma:A_variation}
If $J_{\xi}(\kappa_0)=0$ for some $\xi\in[m,m+1]$ then $A_m(\kappa_0)\le A_{m+1}(\kappa_0)$.
\end{lemma}
\begin{proof}
The recursion formula~\eref{equation:recursion} implies 
\begin{eqnarray*}
\hspace{-23mm}A_m(\kappa_0)^2-A_{m+1}(\kappa_0)^2&=&J_m(\kappa_0)^2+J_{m+1}(\kappa_0)^2-\frac{2m}{\kappa_0}J_m(\kappa_0)J_{m+1}(\kappa_0)\\&-&J_{m+1}(\kappa_0)^2-J_{m+2}(\kappa_0)^2+\frac{2(m+1)}{\kappa_0}J_{m+1}(\kappa_0)J_{m+2}(\kappa_0)\\&=&\frac{2(m+1)}{\kappa_0}J_{m+1}(\kappa_0)(J_m(\kappa_0)-J_{m+2}(\kappa_0))\\&-&\frac{2m}{\kappa_0}J_m(\kappa_0)J_{m+1}(\kappa_0)+\frac{2(m+1)}{\kappa_0}J_{m+1}(\kappa_0)J_{m+2}(\kappa_0)\\&=&\frac{2}{\kappa_0}J_m(\kappa_0)J_{m+1}(\kappa_0).
\end{eqnarray*}
For any fixed positive argument $x$, the function $\mathbb{R}\ni\mu\mapsto J_{\mu}(x)$ is differentiable and not identically zero. Thus, by assumption, and by lemma~\ref{lemma:only_one_zero}, this function changes sign precisely once in the interval $[m,m+1]$, so
\begin{equation*}
A_m(\kappa_0)^2-A_{m+1}(\kappa_0)^2=\frac{2}{\kappa_0}J_m(\kappa_0)J_{m+1}(\kappa_0)\le0.
\end{equation*}
\end{proof}

\begin{remark}\label{remark:h2_increasing}
Clearly, the function $\mathbb{N}_0\ni m\mapsto |H_m^{(1)}(\kappa)|^2$ is positive-valued. It is also strictly increasing, as can be seen from Nicholson's integral for $|H_m^{(1)}(\kappa)|^2$~\citep[pp. 441-444]{Watson-1945},
\begin{equation*}
|H^{(1)}_m(\kappa)|^2=\frac{8}{\pi^2}\int_{t=0}^{\infty}K_0(2\kappa\sinh t)\cosh2mt,\quad m\in\mathbb{Z},\,\kappa>0,
\end{equation*}
where $K_0$ is the modified Bessel function of the second kind. This namely implies
\begin{equation*}
\partial_m(|H_m^{(1)}(\kappa)|^2)=\frac{16m}{\pi^2}\int_{t=0}^{\infty}K_0(2\kappa\sinh t)\sinh2mt>0,\quad m>0,
\end{equation*}
since both $K_0$ and the hyperbolic sine are positive over positive reals.
\end{remark}
The above discussion suffices for a proof of the lower bound $\BW_-$.
\begin{proof}[\textbf{Proof of Theorem~\ref{theorem:MAIN}}]\label{proof:M-}
Let $m\in\mathbb{N}_0$. If $\kappa_0>j_{m,1}$ then there are $n\in\mathbb{N}_0$ and $\xi\in[m+n,m+n+1]$ satisfying $J_{\xi}(\kappa_0)=0$, and, by Lemma~\ref{lemma:A_variation}, $A_{m+n}(\kappa_0)\le A_{m+n+1}(\kappa_0)$. Since $\mathbb{N}_0\ni\mu\mapsto|H_{\mu}^{(1)}(\kappa)|$ is strictly increasing, and $\sigma_{\mu}$ is proportional to $|H_{\mu}^{(1)}(\kappa)|A_{\mu}(\kappa_0)$ for $\mu\in\mathbb{N}_0$, we have $\sigma_{m+n}\le\sigma_{m+n+1}$, hence $m<\mathscr{B}$. In conclusion, $\mathscr{B}\ge{\rm argmax}_{m\in\mathbb{N}_0}\{j_{m,1}<\kappa_0\}+1={\rm argmin}_{m\in\mathbb{N}_0}\{j_{m,1}\ge\kappa_0\}$.
\end{proof}

\begin{proof}[\textbf{Proof of Corollary~\ref{corollary:ceil}.}]
We use that~\citep[p. 516]{Watson-1945} $j_{m,1}=m+a_-m^{1/3}+O(m^{-1/3})$, with $a_-=1.855757$. The real solution of $n^3+a_-n-\kappa_0=0$ is readily found to be
\begin{equation*}
\hspace{-4mm}n=\frac{1}{6}\left(108\kappa_0+12\sqrt{12a_-^3+81\kappa_0^2}\right)^{1/3}-\frac{2a_-}{\left(108\kappa_0+12\sqrt{12a_-^3+81\kappa_0^2}\right)^{1/3}},
\end{equation*}
so $\BW_-\approx\lceil n^3\rceil$ for sufficiently large $\kappa_0$.
\end{proof}
For completeness, let us also provide an approximate expression for the conjectured value of the upper bound $\BW_+$. We have~\citep[p. 516]{Watson-1945} $y_{m,1}=m+a_+m^{1/3}+O(m^{-1/3})$, with $a_+=0.931577$. Also, $m<m-2+a_+(m-2)^{1/3}$ for integer $m\ge12$, so, for sufficiently large $\kappa_0$, $m>\kappa_0$ implies $y_{m-2,1}>\kappa_0$ and hence $\BW_+\approx\widetilde{\mathscr{B}}_+=\lceil\kappa_0\rceil$.

\section{Numerical validation}\label{section:numericalvalidation}

We here compute the bandwidth $\BW$, as well as the bandwidth bounds $\BW_-$ and $\BW_+$ of Theorem~\ref{theorem:MAIN} and Conjecture~\ref{conjecture:y}, respectively, for 300 values of the size parameters $\kappa=\kappa_0$ uniformly distributed over the interval $\kappa\in[2,100\pi]$. Recall that $\kappa=kR=2\pi R/\lambda$ and $\kappa_0=kR_0=2\pi R_0/\lambda$, where $R$ is the radius of the sampling circle $\partial D$, $R_0$ is the radius of the source domain, and $\lambda$ is the operating wavelength. Thus, we consider 300 values of the relative wavelength $\lambda/R=\lambda/R_0$ distributed nonuniformly over the interval $\lambda/R\in[1/50,\pi]$. Figure 5 shows the errors $\varepsilon_{\pm}=\BW_{\pm}-\BW$ and the relative errors $\varepsilon_{\sf rel,\pm}=|\BW_{\pm}-\BW|/\BW$ in the estimated bandwidth as function of the problem size parameter $\kappa$. 
\begin{figure}\label{figure:epsilon}
\begin{center}
\includegraphics[width=0.495\textwidth]{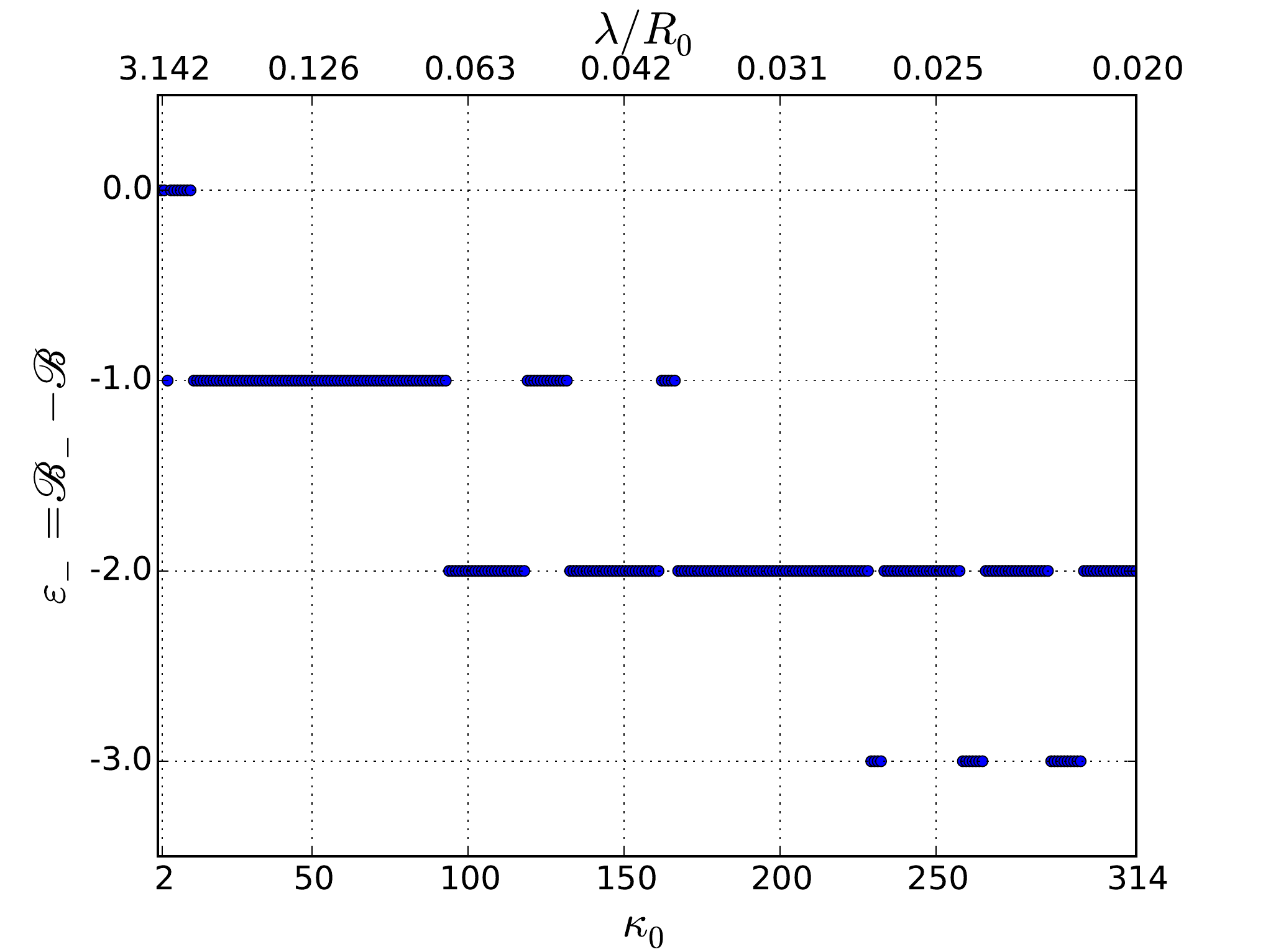}
\includegraphics[width=0.495\textwidth]{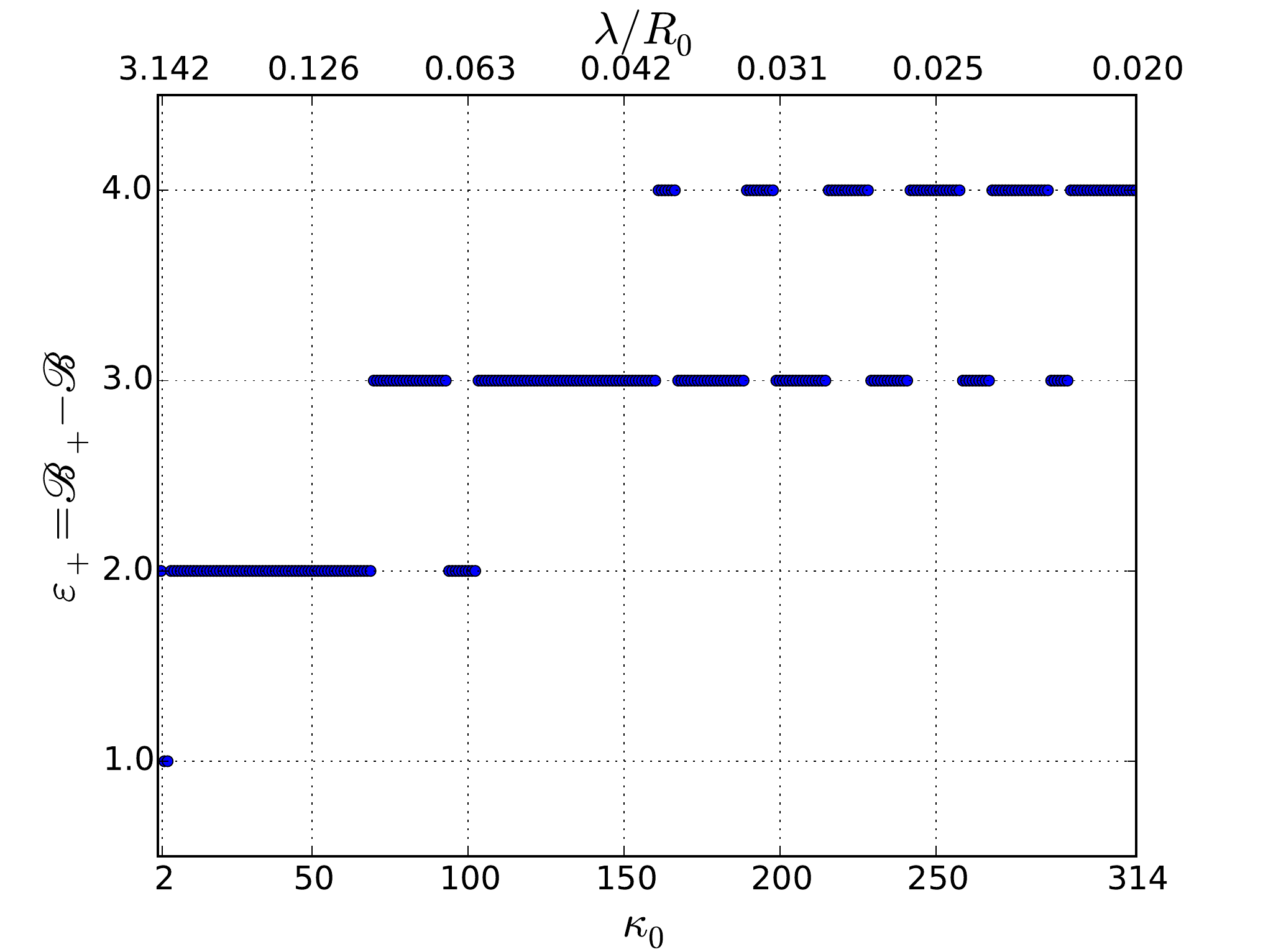}
\includegraphics[width=0.495\textwidth]{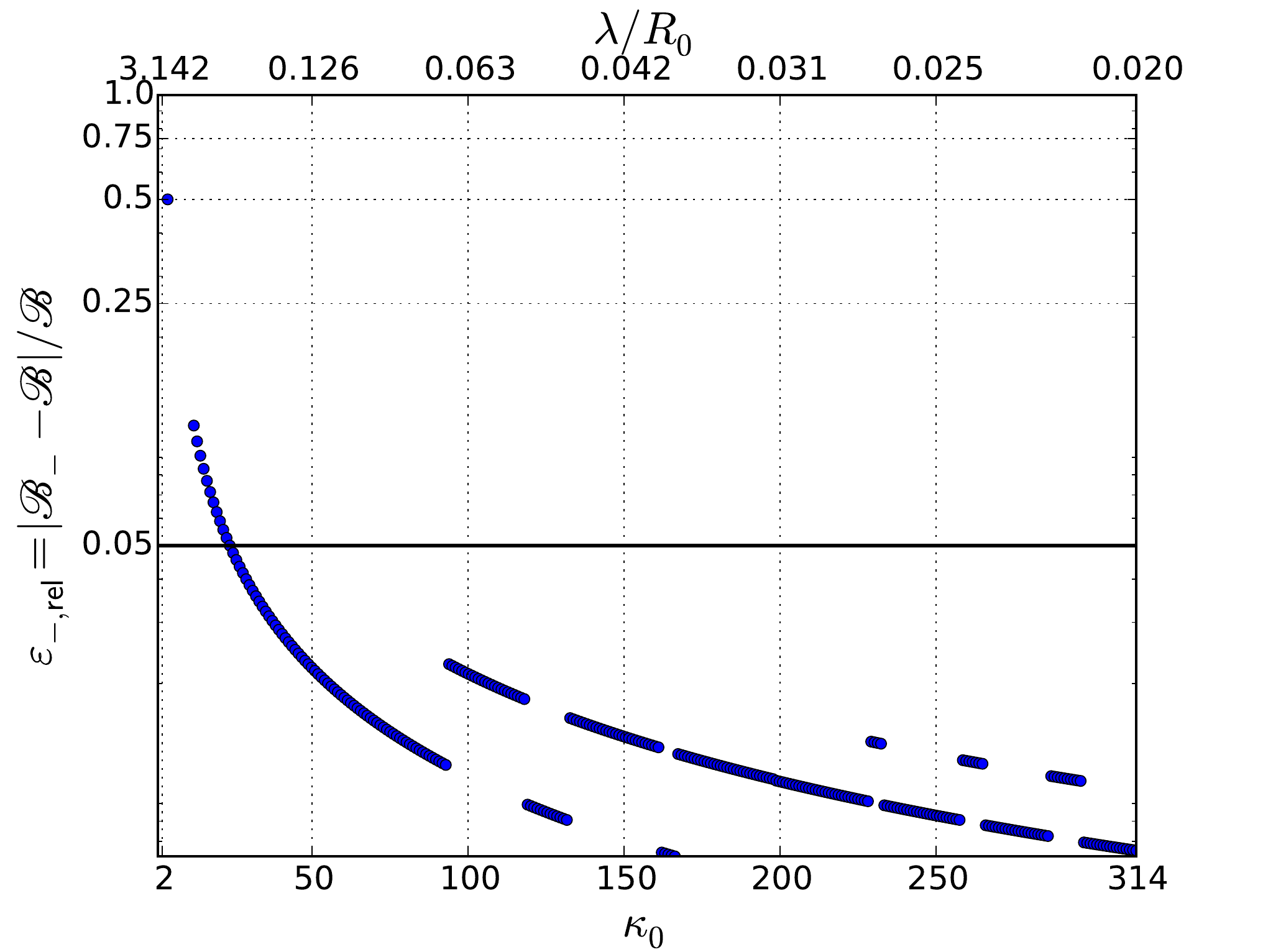}
\includegraphics[width=0.495\textwidth]{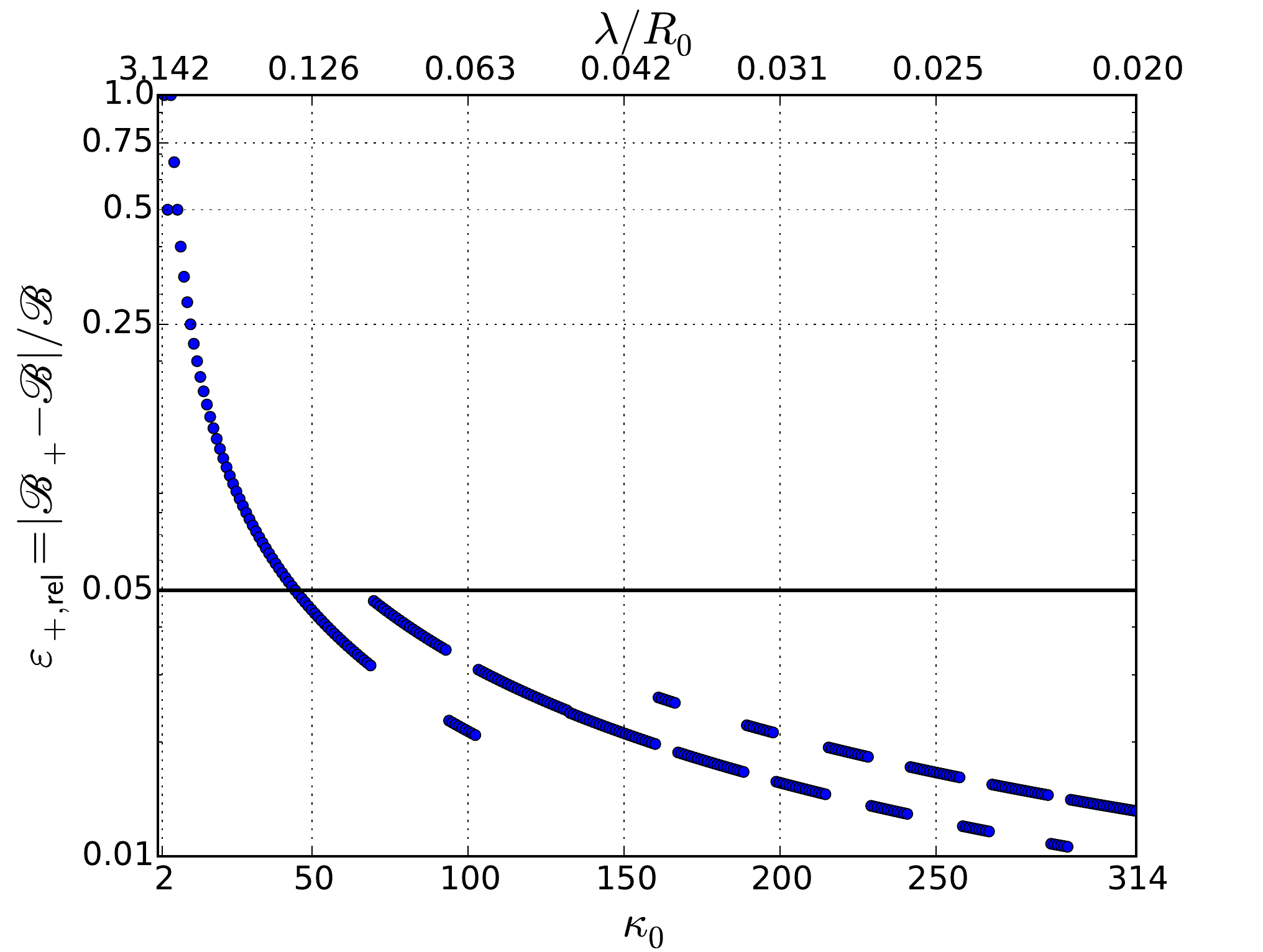}
\caption{Errors (top) and relative errors (bottom) in the lower and upper bounds on the bandwidth $\BW$ of Theorem~\ref{theorem:MAIN}, over the range of the source size parameter corresponding to $R/\lambda=R_0/\lambda\in[1/10,50]$.}
\end{center}
\end{figure}
\begin{figure}\label{figure:epsilon_approx}
\begin{center}
\includegraphics[width=0.495\textwidth]{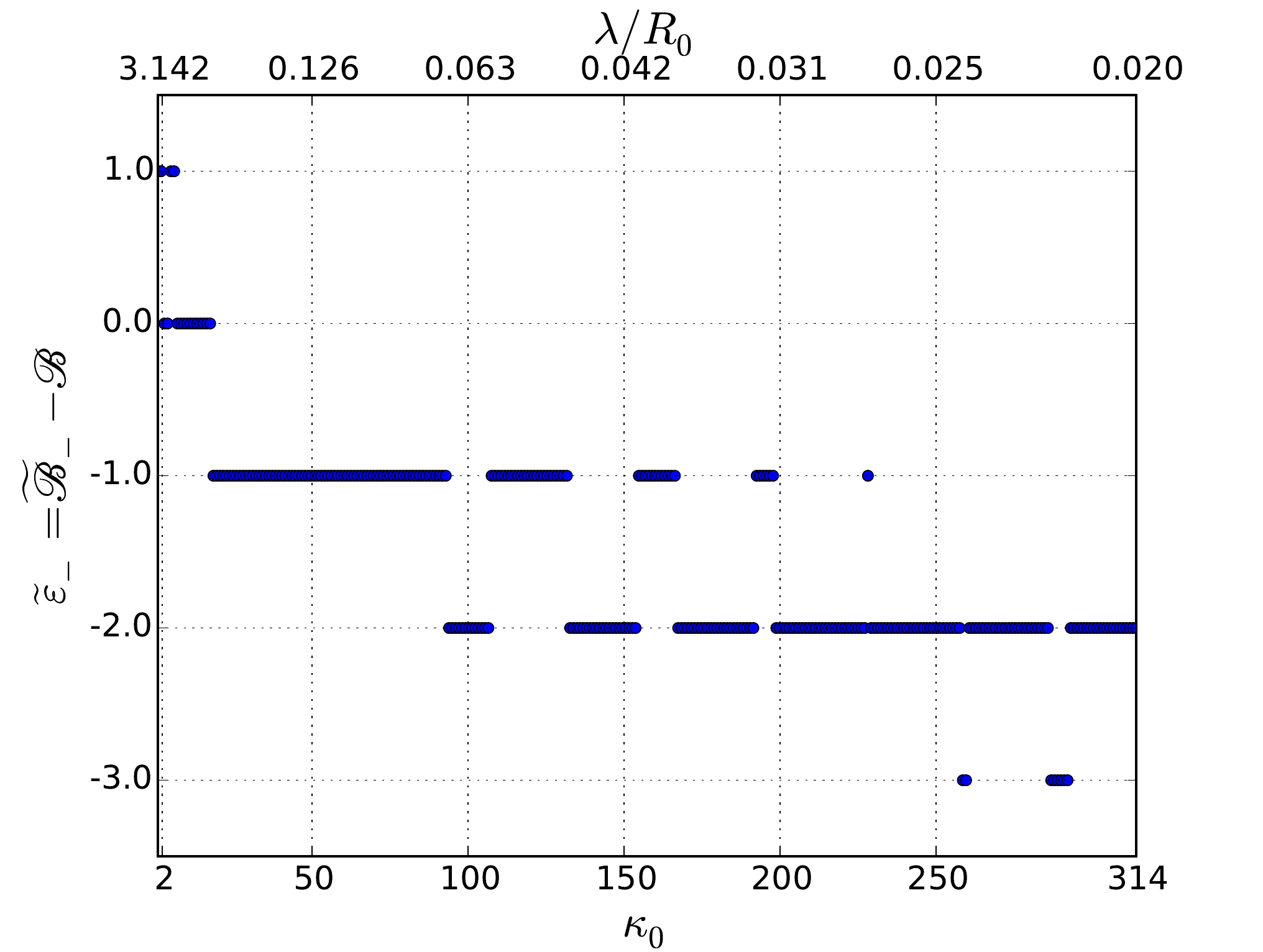}
\includegraphics[width=0.495\textwidth]{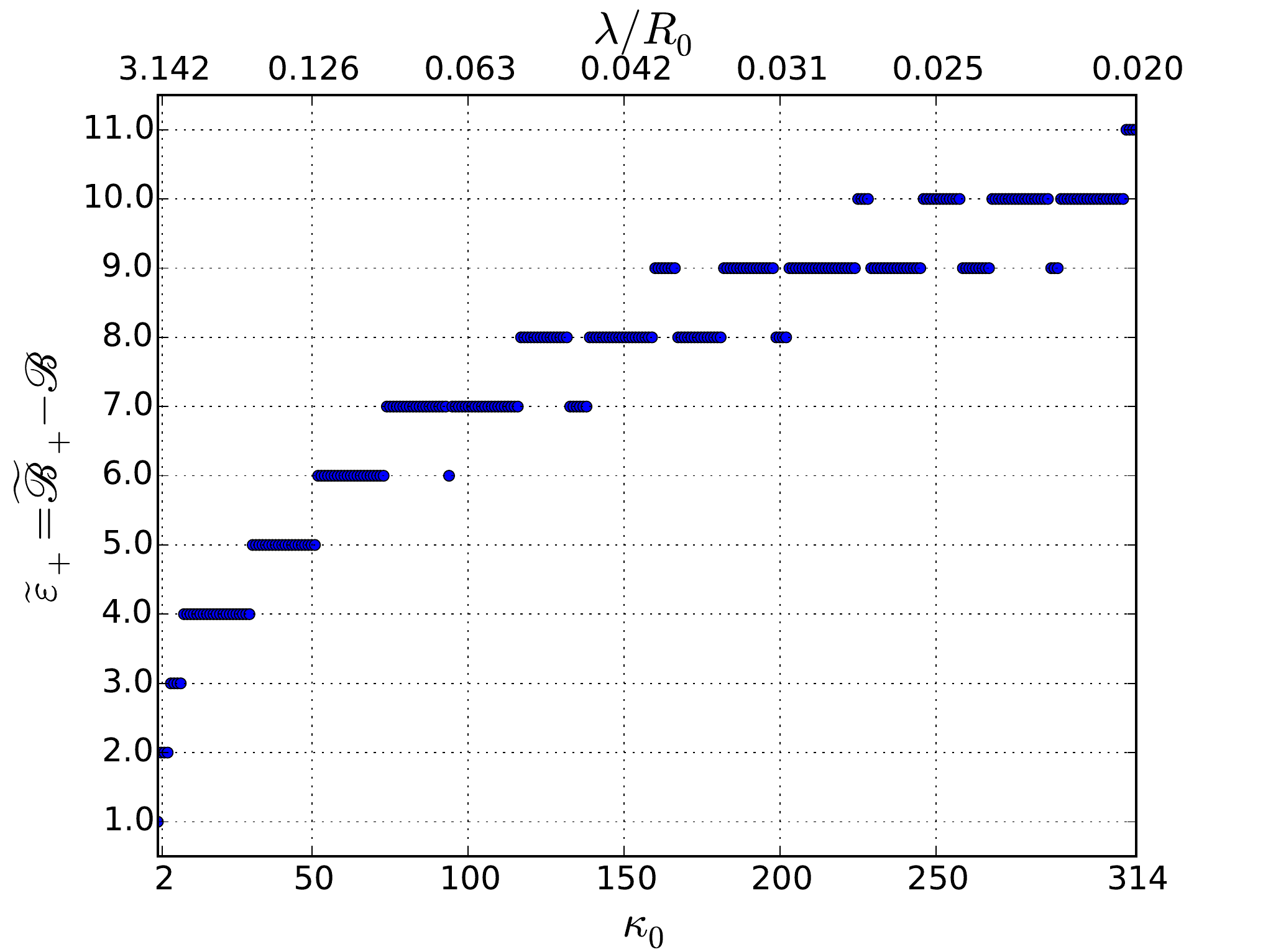}
\includegraphics[width=0.495\textwidth]{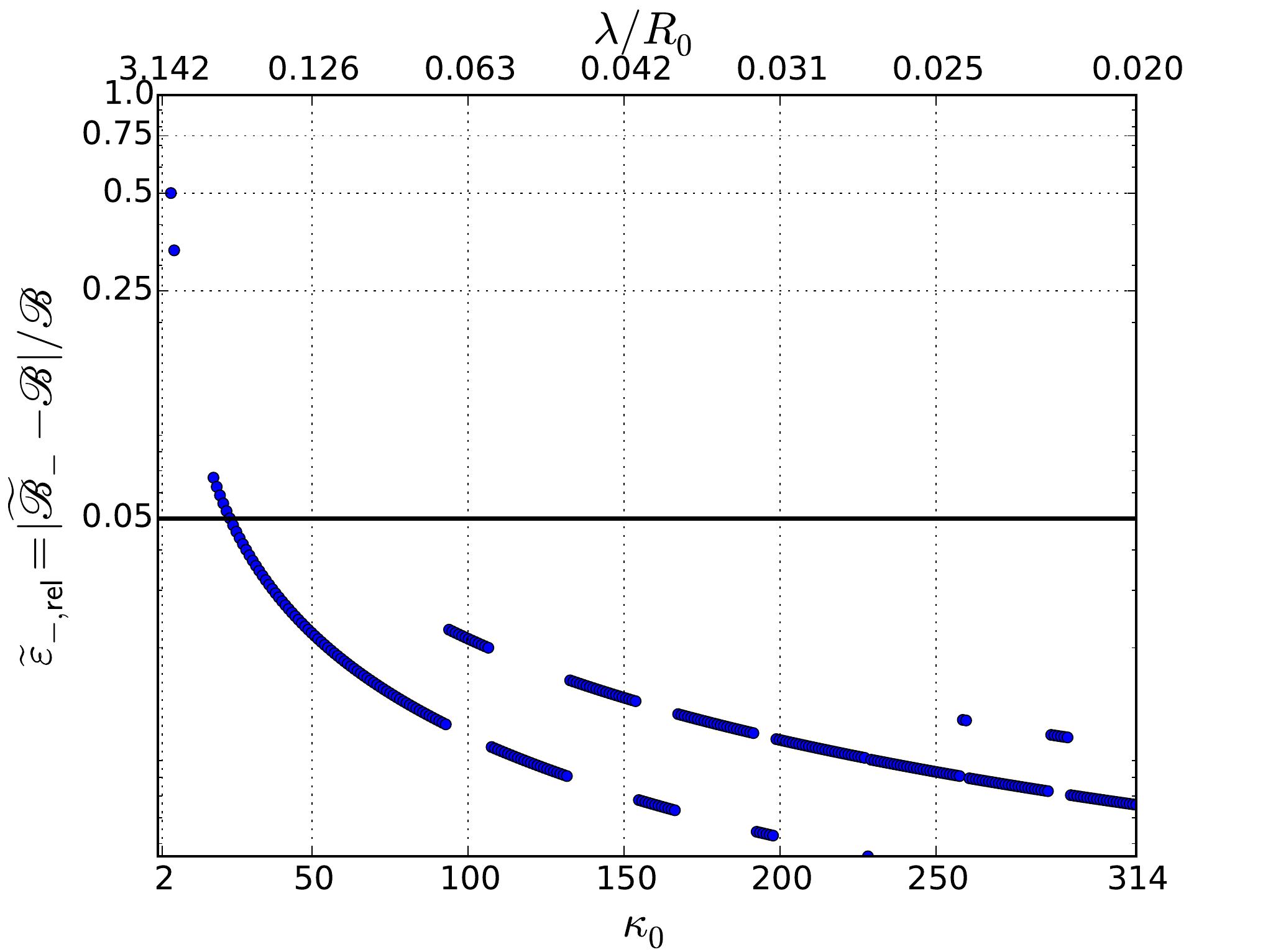}
\includegraphics[width=0.495\textwidth]{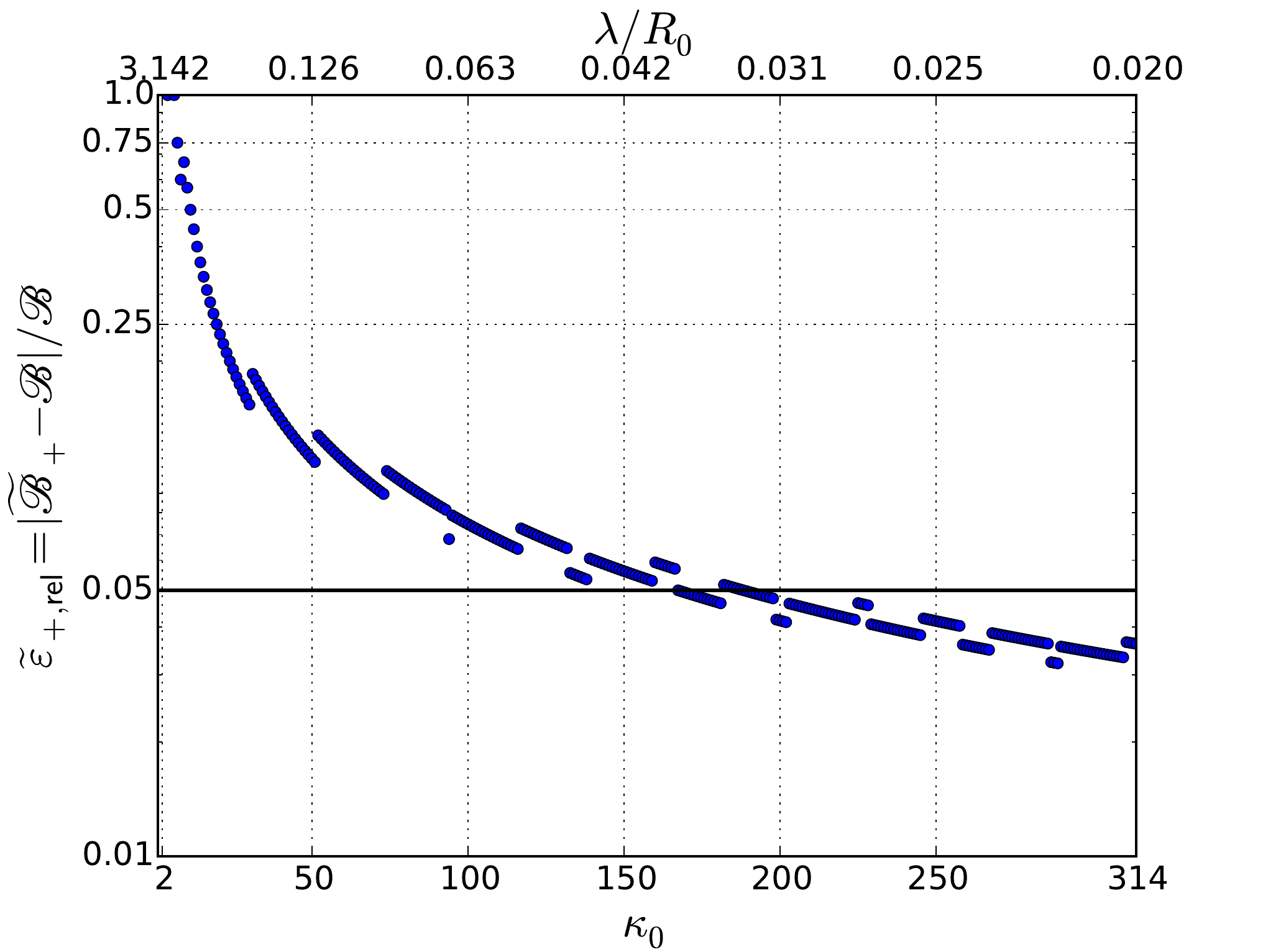}
\caption{Errors (top) and relative errors (bottom) in the approximate lower and upper bounds $\widetilde{\BW}_{\pm}$ on the bandwidth $\BW$ of Theorem~\ref{theorem:MAIN}, over a range of the problem size parameter $\kappa$.}
\end{center}
\end{figure}
For the two lowest considered values of $\kappa$, we find that $\BW=0$ and $\BW_-=0$; there, we set $\varepsilon_{\sf rel,-}=0$. Both $\BW$ and $\mathscr{B}_-$ are positive for higher considered values of $\kappa$. In particular, there is zero bandwidth for $\kappa$ smaller than some threshold value between approx. 1.7 and approx. 2.7, and for such size parameters $\kappa$ the inverse source problem is, from the viewpoint of the bandwidth of the singular values, similar to the inverse heat conduction problem. Over the considered interval for $\kappa$, the mean errors are $\overline{\varepsilon}_-=-1.68$, $\overline{\varepsilon}_+=3.02$, and the maximum absolute errors are $\max|\varepsilon_-|=3$, $\max|\varepsilon_+|=4$. The relative error in $\BW_-$ is below $5\%$ for $\kappa\ge24.7461$, i.e., for $R/\lambda\ge3.94$, and $\BW_+$ is below $5\%$ for $\kappa\ge45.7181$, i.e., for $\lambda/R\ge7.28$. 

We find both $\BW$, $\BW_-$ and $\BW_+$ to be approximately linear functions of $\kappa$ in the given interval, with least-squares fits summarized in Table~\ref{table:lsf}.
\begin{table}
\begin{center}
\begin{tabular}{c|c|c|c}\label{table:lsf}
                         & linear interpolant                        & mean absolute error & standard deviation\\
                         \hline
$\BW$             & $0.9793\kappa-3.9569$ & $0.4813$ & $3.9\cdot10^{-4}$\\
\hline
$\BW_-$ & $0.9736\kappa-4.7394$ & $0.5715$ & $4.8\cdot10^{-4}$\\
\hline
$\BW_+$ & $0.9861\kappa-2.0083$ & $0.4052$ & $3.4\cdot10^{-4}$
\end{tabular}
\end{center}
\caption{Linear regression of the computed bandwidth $\mathscr{B}$, and the lower ($\BW_-$) and upper ($\BW_+$) bandwidth bounds. We have here held equal the size parameters $\kappa$ and $\kappa_0$.}
\end{table}
Figure 6 shows errors in the approximations $\widetilde{\BW}_{\pm}$ (for definition of $\widetilde{\BW}_{\pm}$ see proof of Corollary~\ref{corollary:ceil} and the paragraph immediately following it, on page 9.) The approximate expression for the lower bound shows almost the same small error as the lower bound itself, and the approximate expression $\widetilde{\mathscr{B}}_+\approx\lceil\kappa_0\rceil$ for the upper bound, while simple, has error below $5\%$ only for problem size parameters of approx. 175 or higher when $\kappa=\kappa_0$ is maintained.

Our bounds $\BW_{\pm}$ are independent of the radius of the measurement surface, and we next validate this property numerically. Figure 7 shows the first 71 nonnegative-index singular values of the forward operator $F$ with size parameters $\kappa=100\pi$, $\kappa_0=10\pi$. The bandwidth is unchanged at $\BW=27$ (compare with Figure 3), as predicted by our bounds. The decrease in the numerical stability of the ISP due to the measurement boundary being farther away from the source is instead expressed in terms of the overall lower level of the singular values.
\begin{figure}
\begin{center}
\includegraphics[width=0.495\textwidth]{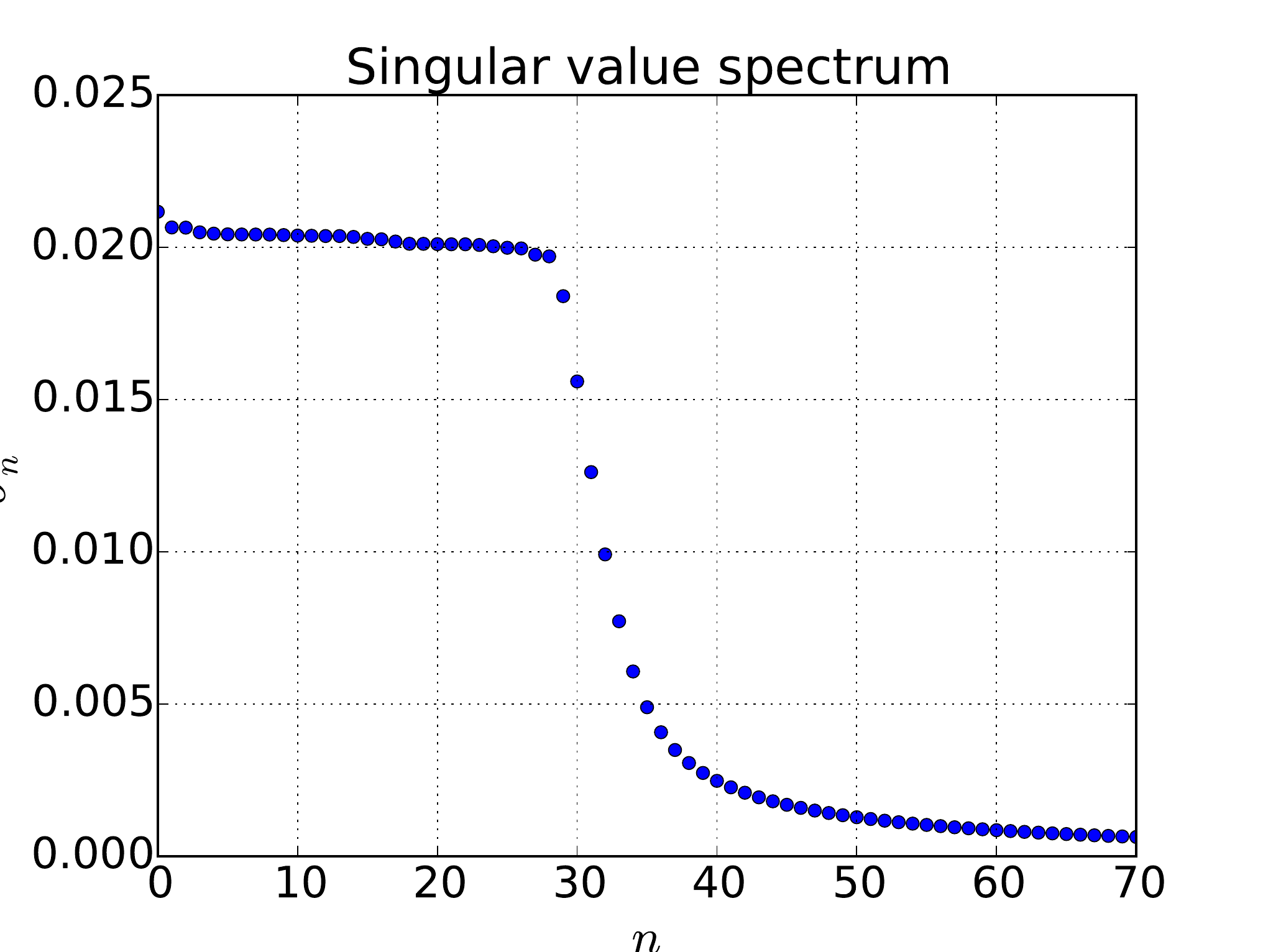}
\includegraphics[width=0.495\textwidth]{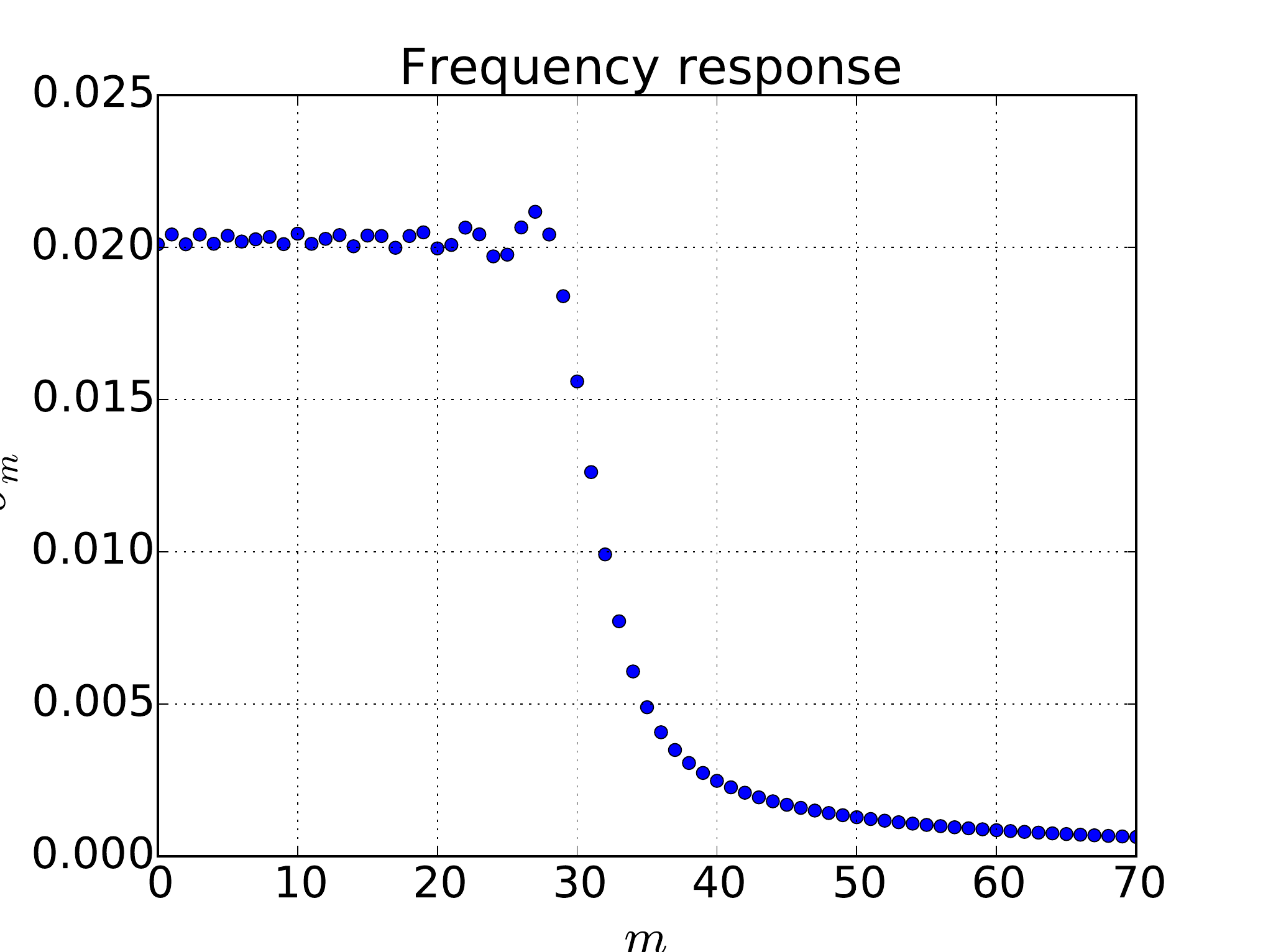}
\caption{Part of the singular value spectrum of the forward operator $F$ for $\kappa=100\pi$, $\kappa_0=10\pi$. Left: the singular values $\sigma_n$, $0\le n\le70$, ordered in a decreasing sequence. Right: the same singular values ordered according to the angular frequency $m$ of the right singular vector $\phi_m$.}
\end{center}
\end{figure}

\section{Discussion}\label{section:discussion}

The bandwidth estimates $\BW_{\pm}$ are directly applicable as optimal filter estimates in the numerical solution of the inverse source problem in terms of a truncated singular value decomposition (TSVD) of the forward operator. Next, it has been amply observed in the literature concerning the single-frequency inverse source problem that the numerical stability of the solution increases with the operating frequency. Theorem~\ref{theorem:MAIN} confirms and explicitly quantifies this increase in numerical stability, also for non-asymptotic frequencies.

Theorem~\ref{theorem:MAIN} of course has direct implications for the maximum achievable stable resolution of the reconstruction in the inverse source problem. Detailed analysis of this resolution requires an investigation of the pointwise behavior of the left singular vectors of the forward operator. While we here do not perform such analysis, we do note that the left singular vectors tend to be supported near the origin for low values of index $m$, and near the measurement boundary for high index values. This means the amplified noise produces is a 'wall of non-information' near the measurement boundary and blocks faithful reconstruction of the source inside $D$.

As shown in~\cite{Bao-2010} and in Section~\ref{section:spectralanalysis} here, the right singular vectors (defined over the measurement boundary) of the forward operator are proportional to $\exp(\rmi m\theta)$, $m\in\mathbb{Z}$. This means that the bandwidth index $\BW$ is approximately the angular frequency of the highest-frequency data component that can be stably inverted. Thus, the sampling theorem~\citep{Shannon-1949} is directly applicable with Theorem~\ref{theorem:MAIN} to give the following: in case the radiated field $u$ is sampled equidistantly at the boundary $\partial D$, any angular sampling rate greater than approximately $\Delta\theta\approx\pi/\BW\le\pi/\BW_-=\pi/{\rm argmin}_{m\in\mathbb{N}_0}\{j_{m,1}\ge\kappa_0\}$ is excessive due to the limited bandwidth of the forward operator.

Bandwidth bounds in Theorem~\ref{theorem:MAIN} and Conjecture~\ref{conjecture:y} involve the size parameter of only the source support, and in light of the successful numerical validation of these bounds, we find it justified to say that the bandwidth is generally independent of the radius $R$ of the measurement boundary relative to the radius $R_0$ of the source support (as long as $R\ge R_0$). As illustrated in Section~\ref{section:numericalvalidation}, the decrease in the robustness of the inversion (in the presence of noise) as $R_0/R$ decreases seems instead to be expressed by a lower overall level of the singular values. We therefore briefly analyze the asymptotic behavior of the singular spectrum~\eref{equation:xi_m} as $m\rightarrow0$, and as $m\rightarrow\infty$. The standard large-argument approximation of the Bessel functions of the first and second kind, valid for $\kappa_0\gg m^2-1/4$, yields
\begin{equation*}
\hspace{-10mm}J_m(\kappa_0)\sim\sqrt{\frac{2}{\pi\kappa_0}}\cos\left(\kappa_0-\frac{m\pi}{2}-\frac{\pi}{4}\right),\quad Y_m(\kappa_0)\sim\sqrt{\frac{2}{\pi\kappa_0}}\sin\left(\kappa_0-\frac{m\pi}{2}-\frac{\pi}{4}\right),
\end{equation*}
so
\begin{eqnarray*}
A_m(\kappa_0)^2&\sim&\frac{2}{\pi\kappa_0}\left(\cos\left(\kappa_0-\frac{m\pi}{2}-\frac{\pi}{4}\right)^2\right.\\&-&\Biggl.\cos\left(\kappa_0-\frac{m\pi}{2}-\frac{\pi}{4}+\frac{\pi}{2}\right)\cos\left(\kappa_0-\frac{m\pi}{2}-\frac{\pi}{4}-\frac{\pi}{2}\right)\Biggr)\\&=&\frac{2}{\pi\kappa_0}\left(\cos\left(\kappa_0-\frac{m\pi}{2}-\frac{\pi}{4}\right)^2+\sin\left(\kappa_0-\frac{m\pi}{2}-\frac{\pi}{4}\right)^2\right)=\frac{2}{\pi\kappa_0}
\end{eqnarray*}
and, since $\kappa\ge\kappa_0$, we also have $H_m(\kappa)^2\sim2/\pi\kappa$. Thus
\begin{equation*}
\sigma_m\sim\frac{\sqrt{2}}{\pi}\lambda\sqrt{R_0}
\end{equation*}
for $R_0/\lambda\gg(m^2-1/4)/2\pi$. Forward operators mapping from source spaces with larger supports thus have higher-valued singular values in the bandpass region, regardless of the size $R$ of the measurement boundary relative to the size $R_0$ of source support. However, we also see that the height of the bandpass decreases when the operating wavelength lambda decreases (equivalently, when the operating frequency increases), which may counteract the increase in stably recoverable information gained due to the increase in bandwidth. In the small-argument limit ($0<\kappa^2\ll m+1$) the standard approximation is
\begin{equation*}
J_m(\kappa_0)\approx\frac{1}{m!}\left(\frac{\kappa_0}{2}\right)^m,\quad Y_m(\kappa_0)\approx-\frac{(m-1)!}{\pi}\left(\frac{2}{\kappa_0}\right)^m,
\end{equation*}
so (since $\kappa_0\le\kappa$) $A_m(\kappa_0)^2\sim(\kappa_0/2)^{2m}m!^{-2}(m+1)^{-1}$ and $H_m(\kappa)^2\sim(\kappa/2)^{2m}m!^{-2}+m!^2m^{-2}\pi^{-2}(\kappa/2)^{-2m}$, resulting in
\begin{eqnarray*}
A_m(\kappa_0)^2H_m(\kappa)^2&\sim\left(\frac{R_0}{R}\right)^{2m}\frac{1}{\pi^2m^2(m+1)},
\end{eqnarray*}
and thus
\begin{equation*}
\sigma_m\sim\frac{1}{m}\sqrt{\frac{2}{m+1}}\left(\frac{R_0}{R}\right)^{m-1/2}R_0^{3/2}.
\end{equation*}
Evidently, the ratio $R_0/R$ of the source support radius to the measurement boundary radius strongly affects the rate of decay of the singular values, the robustness of the inversion to noise generally improving as the source support approaches the measurement boundary.

\section{Conclusion and further work}\label{section:conclusionand}

We analyzed the singular values of the forward operator associated with the single-frequency inverse source problem for the Helmholtz equation in the plane. In particular, we considered bounds on the information content that is preserved by the forward operator, proving a tight lower bound and conjecturing a tight upper bound on the singular value index of the highest-frequency data component that is stably recoverable. The bounds were expressed in terms of the zeros of Bessel functions of the first and the second kind. We validated both bounds numerically, establishing concrete estimates on the stably recoverable information in the inverse source problem regardless of the data sampling rate and the choice of regularization. The result can be used directly, e.g., to estimate optimal TSVD filters and data sampling rates. 

Proving the statement in Conjecture~\ref{conjecture:y} is a natural next step. Also, it would complete the picture to supplement the results on the bandwidth with a more precise description of the general levels and decay rates of the singular values as function of the size parameters of the source support and of the measurement boundary, individually or in relation to one another. Finally, a spectral analysis of the forward operator in dimension greater than 2 will be interesting.

\bibliographystyle{harvard}

\end{document}